\documentclass{amsart}

\usepackage{graphicx}
\usepackage{fancyhdr}
\usepackage{pst-node}
\usepackage{tikz-cd} 
\usepackage{hyperref}
\usepackage{amsmath}
\newtheorem{theorem}{Theorem}[section]
\newtheorem{lem}[theorem]{Lemma}
\newtheorem{conj}[theorem]{Conjecture}
\newtheorem{assum}[theorem]{Assumption}
\theoremstyle{definition}

\newtheorem{prop}[theorem]{Proposition}

\newcommand{\ind}{\mathrm{index}}
\newcommand{\ed}{\mathrm{End}}
\newcommand{\map}{\mathrm{Map}}
\newcommand{\dir}{\mathrm{D}}
\newcommand{\Dir}{\mathcal{D}}
\newcommand{\sign}{\mathrm{Sign}}

\newcommand{\vol}{\mathrm{Vol}}
\theoremstyle{remark}

\numberwithin{equation}{section}

\begin{document}

\title{Seiberg-Witten-Casson Invariant of Homology $S^1\times S^3$ with Circle Action}

\author{Daoyuan Han}
\address{Department of Mathematics, Brandeis University, Waltham, Massachusetts 02453}
\curraddr{Department of Mathematics, Lehigh University, Bethlehem, Pennsylvania 18015}
\email{dah517@lehigh.edu}



\date{}



\begin{abstract}
In this paper we shall compute the Mrowka-Ruberman-Saveliev invariant introduced in \cite{liftof} for the case when the manifold admits a free circle action. 
\end{abstract}

\maketitle

\section{Introduction}
The Mrowka-Ruberman-Saveliev invariant \cite{liftof} defined for 4-manifolds with $b_2^+=0$ is the count of irreducible solutions plus an index of Dirac operator over end-periodical manifold. Several special cases have been computed by others, for example, when X is of the form $S^1\times Y$ and $Y$ is of the homology type of the three sphere $S^3$, Mrowka, Ruberman, Saveliev proved using Lim's work in \cite{equiofseib} that it coincides with Casson's invariant. The authors of \cite{liftof} computed this invariant in other cases, like mapping tori. Moreover, we want to verify a special case of the conjecture made in the paper \cite{liftof}, which states that the Mrowka-Ruberman-Saveliev invariant is the same as the invariant define by Furuta and Ohta in \cite{foinvariant}, which is considered as another generalization of Casson's invariant to 4-manifold.

We shall begin this section by reviewing the definition of the Mrowka-Ruberman-Saveliev invariant introduced in \cite{liftof}. Let $X$ be an integral homology $S^1\times S^3$, then define 
\[
\lambda_{SW}(X)=\#\mathcal{M}(X,g,\beta)-\omega(X,g,\beta),
\]
where $\#\mathcal{M}(X,g,\beta)$ is the count of irreducible solutions in the Seiberg-Witten moduli space over $X$ equipped with metric $g$ and perturbation $\beta$ and $\omega(X,g,\beta)$ is a correction term so that $\lambda_{SW}(X)$ is independent of $g$ and $\beta$. The correction term $\omega(X,g,\beta)$ is defined as
\[
\omega(X,g,\beta)=\ind \Dir^{+}(Z_+,g,\beta)+\text{sign}(Z)/8,
\]
where $Z_+=Z\cup W_1\cup W_2\cup\cdots$, with each copy of $W_i$ a cobordism formed by cutting along a 3-submanifold $M$ representing generator of $H_3(X)$, and $Z$ is a spin 4-manifold with boundary $M$. $\Dir^+(Z_+,g,\beta)$ is the Dirac operator over $Z_+$ equipped with $Spin$-structure extending that over $W$ to $Z$. And it is proved in \cite{liftof} that this correction term is independent of $Z$ and the way to extend the Spin structure. 

This invariant can be treated as a lift of Rohlin's invariant by Theorem A in \cite{liftof}, which can also be considered as a generalization of Theorem 1.2 by Chen in \cite{weimin}, where an integer invariant $\alpha(Y)$ for homology sphere $Y$ is defined and equal to Rohlin's invariant $\text{mod }  2$. The Chen's invariant in \cite{weimin} is also defined as a combination of Seiberg-Witten invariant and index correction term. 

The proof that $\lambda_{SW}$ is well-defined in \cite{liftof} uses the blown-up of SW-equation and shows at first that for generic metric and perturbation $(g,\beta)$, the pair is regular, meaning that the corresponding blown-up moduli space has no reducible solution. The moduli space under regular pair of $(g,\beta)$ is a zero dimensional manifold by computing the virtual dimension. Considering a path of such regular pairs, it's proved that the corresponding parametrized  moduli space is a 1-dimensional manifold with boundary $\mathcal{M}(X,g_0,\beta_0)\cup\mathcal{M}(X,g_1,\beta_1)\cup\mathcal{M}^0_I$ where $\mathcal{M}^0_I$ denotes the path components approaching reducibles. Thus the change of Seiberg-Witten invariants can be expressed as the count of points in $\mathcal{M}^0_I$. Then the remainder of the proof shows that there is a 1-1 correspondence between the change of correction terms along the same path $(g_t,\beta_t)$ and $\mathcal{M}^0_I$. It is achieved by expressing the change of correction terms as a spectral flow of certain path of Dirac operators over $X$. This new path of Dirac operators over $X$ is derived from Laplace-Fourier transform of the end-periodic Dirac operators over $Z_+$. Note that the spectral flow only changes when a Dirac operator on the path has nontrivial kernel and it remains to be seen that the parameters where the kernel is nontrivial are in 1-1 correspondence with the points in $\mathcal{M}^0_I$.

When the manifold admits a free circle action, the Seiberg-Witten invariant and the correction term can be computed more explicitly. We assume that the circle action induces a $S^1$-bundle $\pi: X\rightarrow Y$ whose Euler number $e=1$. The submanifold $M$ of $X$ representing the homology generator of $H_3(X; \mathbb{Z})$, fibers over a 2-surface $\Sigma$. It is shown by Baldridge \cite{baldridge} that in this case, the Seiberg-Witten invariant of $X$ can be related to the 3 dimensional Seiberg-Witten invariant of $Y$, and this 3-dimensional Seiberg-Witten invariant can be further related to Alexander polynomial of a knot, surgery on which of $S^3$ gives $Y$. The correction term on the other hand, can be computed by using a special neck-stretching metric in \cite{split} over $Z_+$ with the effect of stretching $M\times[0,R]$ by letting $R$ be sufficiently large, and we can then use the index formula for cylindrical end manifold to compute the correction term. Thus we have the following theorem, 

\begin{theorem}\label{main}
Let $X$ be a smooth 4-manifold of integral homology $S^1\times S^3$ with a free circle action such that $X$ is circle bundle over $Y$ and $H_*(Y)=H_*(S^1\times S^2)$.  Then there exists a pair $(g_X,\beta)$ such that $\#\mathcal{M}(X,g_X,\beta)=\Delta_Y''(1)$, where $\Delta_Y(t)$ denotes the normalized Alexander polynomial and the correction term $\omega(X,g_X,\beta)=0$. 
\end{theorem}

Note that when the infinite cyclic cover $\widetilde{X}$ has the same homology as $S^3$,  we have $\#\mathcal{M}(X,g,\beta)=0$ as $\Delta_Y(t)$ is trivial. Thus we can verify the following conjecture made in \cite{liftof} in the case when $X$ admits a free circle action. 
\begin{conj}[\cite{liftof}]
For any smooth oriented homology oriented $4$-manifold $X$ with the $\mathbb{Z}[\mathbb{Z}]$-homology of $S^1\times S^3$, one has 
\[
\lambda_{SW}(X)=-\lambda_{FO}(X).
\]
\end{conj}
Here the $\lambda_{FO}(X)$ denotes the Furuta-Ohta invariant introduced in \cite{foinvariant} defined by counting the points in the moduli space of irreducible ASD connections on a trivial $SU(2)$ bundle $P\rightarrow X$ and is zero when the assumptions in Theorem \ref{main} are satisfied \cite{cassontype}. 

\section{Seiberg-Witten invariant}
\subsection{Moduli space over circle bundle}
Let $X$ be a smooth 4-manifold admitting a free circle action and the circle bundle $\pi: X\rightarrow Y$ has Euler number 1. We can equip $X$ with a metric of the form $g_X=\eta\otimes\eta\oplus\pi^*g_Y$ where $g_Y$ is a any metric on $Y$ and $i\eta$ is a connection 1-form of the circle bundle $\pi: X\rightarrow Y$. Under these settings, Scott Baldridge proved in \cite{baldridge} that the $Spin^c$-structures $\xi$ for which $SW_X(\xi)\neq 0$ are pulled back from the ones on $Y$ and the moduli space of $Y$ equipped with the metric $g_Y$ is homeomorphic (or orientation preserving diffeomorphic for well chosen metric and perturbation) to a component of the moduli space of $X$ equipped with the metric $g_X=\eta\otimes\eta\oplus\pi^*g_Y$. 

\begin{theorem}(Baldridge \cite{baldridge})\label{bal}
The pullback map induces a homeomorphism
\[\pi^*: \mathcal{M}^*(Y,g_Y,\delta)\rightarrow \mathcal{N}^*(X,g_X,\pi^*(\delta)^+).\] There exists pairs $(g_Y,\delta)$ such that the two moduli spaces are smooth and $\pi^*$ is an orientation-preserving diffeomorphism. 
\end{theorem}

To get an idea of the proof of this theorem, we consider the projection map $\pi: X\rightarrow Y$ which induces a map between moduli space $\pi^*: \mathcal{M}^*(Y,g_Y)\rightarrow\mathcal{M}^*(X,g_X)$. Given a proper perturbation 2-form $\delta$ on $Y$, and pull-back perturbation 2-form $\pi^*(\delta)$ then  $\pi^*: \mathcal{M}^*(Y,g_Y,\delta)\rightarrow\mathcal{M}^*(X,g_X,\pi^*(\delta)^+)$ is a  map between smooth moduli spaces. The map $\pi^*$ is injective. This can be seen by considering two pairs of solutions of Seiberg-Witten equation over $Y$ which are pulled back to solutions over $X$, $(A,\Phi)$, $(A',\Phi')$ which differ by a gauge transformation $g\in \map(X,S^1)$. It remains to check that $g$ is a pull-back from a gauge transformation $g'\in \map(Y, S^1)$. It's not hard to see $g$ can be viewed as a section of $\pi^*(\ed(\det (S)))$ where $S$ is the spinor bundle over $Y$. The connection $\nabla^{\ed}$ on the bundle $\ed(\pi^*(W))$ satisfies 
\[
(\nabla^{\ed}_T) g(\Phi)=\nabla^{A}_T(g\Phi)-g\nabla^A_T(\Phi)=0,
\]
where $T$ is a vertical vector field of unit length along the fiber, because $g\Phi=\Phi'$ is a pull-back from spinor over $Y$. By ellipticity of the first Seiberg-Witten equation $D_A\Phi=0$ as a function of $\Phi$, we know $\Phi\neq 0$ on a dense open subset. Therefore $\nabla^{\ed}_T g=0$ meaning $g$ is constant along the fiber. This shows $g$ is a pull-back from a gauge transformation $g'\in \map(Y,S^1)$. The above argument is due to Baldridge in \cite{baldridge}.

As in \cite{baldridge}, the image of $\pi^*: \mathcal{M}^*(Y,g_Y,\delta)\rightarrow\mathcal{M}^*(X,g_X,\pi^*(\delta)^+)$ is denoted by $\mathcal{N}^*(X,g_X, \pi^*(\delta)^+)$, which is the component in $\mathcal{M}^*(X,g_X, \pi^*(\delta)^+)$ with $Spin^c$ structures pulled back from $Y$. To prove that $\pi^*$ is a diffeomorphism, we need a description of the tangent space to the moduli space at a solution $\mathcal{S}_0$. This is done by considering the deformation complex at $S$ and identifying the tangent space to $\mathcal{S}$ with $\mathcal{H}^1_{\mathcal{S}}$, the first cohomology group of the complex. It's proved in \cite{baldridge} that $\pi^*(\mathcal{H}^1_{\mathcal{S}_0})=\mathcal{H}^1_{\mathcal{S}}$ where $\mathcal{S}$ is an irreducible solution over $X$ and $\mathcal{S}_0$ is a solution over $Y$ such that $\mathcal{S}=\pi^*(\mathcal{S}_0)$. In addition, we can see that $\pi: X\rightarrow Y$ preserves the homology orientation. Given an ordered base of $H^1(Y;\mathbb{R})$, we can use Gysin sequence to see that $H^1(X;\mathbb{R})$ is isomorphic to $H^1(Y;\mathbb{R})$, so an orientation in $H^1(Y;\mathbb{R})$ gives one in $H^1(X;\mathbb{R})$. Note that the homology orientation for $X$ is an orientation for the vector space $H^1(X;\mathbb{R})\oplus H^+(X;\mathbb{R})=H^1(X;\mathbb{R})$ when $X$ is a homology $S^3\times S^1$.

\subsection{Seiberg-Witten invariant for $b_1(Y)=1$}
The Baldridge theorem above helps us understand SW-invariants over total space of circle bundle in terms of those over the base space. In this subsection, we will focus on the 3-dimensional SW-invariant over the base space of the circle bundle $\pi: X\rightarrow Y$. Note that the Baldridge theorem has no restriction on $b_1$. When $b_1(Y)>1$, the SW-invariant is a diffeomorphism invariant while in the case when $b_1(Y)=1$, there is a chamber structure and we have two invariants $SW^{\pm}_Y$ and they are related by the following fundamental wall-crossing formula by Meng and Taubes

\begin{theorem}(Meng-Taubes \cite{mengtaubes})
Let $Y$ be the homology $S^2\times S^1$ obtained from 0-framed surgery on a knot $K\subset S^3$. Then 
\[
SW^{-}_{Y}\cdot (t-t^{-1})^2=\Delta_K(t^2),
\]
where $t=t_T$ for the generator $T$ of $H^2(Y;\mathbb{Z})=\mathbb{Z}$ satisfying $T\cdot\lambda=1$.  
\end{theorem}

When $Y$ is homology $S^2\times S^1$, there is no torsion element in $H_*(Y)$, the $spin^c$-structures $\mathfrak{s}$ over $Y$ are classified by $c_1(\mathfrak{s}):=c_1(det(S))\in H^2(Y,\mathbb{Z})$. We know $c_1(\mathfrak{s})$ is an even class for it is an  integral lift of Stiefel -Whitney class $w_2$. So there is a 1-1 correspondence between $k\in\mathbb{Z}$ and $spin^c$-structures $\mathfrak{s}_k$ with $c_1(\mathfrak{s}_k)=2k$. The pullback $spin^c$-structure $\pi^*\mathfrak{s}_k$ over $X$ are equivalent if $X$ is homology $S^3\times S^1$, we will denote this unique $spin^c$ structure by $\xi_0$. In view of Theorem \ref{bal}, the Seiberg-Witten invariant of the $spin^c$-structure $\xi_0$ over $X$ is equal to the sum of the invariants $SW_Y(\mathfrak{s}_k)$ over all the $spin^c$-structures on $Y$. 

In general, there is a small-perturbation Seiberg-Witten invariant defined for 3-manifold $Y$ with $b_1(Y)=1$. It is defined using Seiberg-Witten equation with an exact perturbation. In the case when $b_1(Y)=1$, the existence of reducible solution gives $F_A=\delta$ where $\delta$ is the perturbation 2-form. This condition gives a codimension 1 "wall" in $H^2(Y;\mathbb{R})$ since it's equivalent to $(2\pi c_1(\mathfrak{s})+\delta)\cdot \lambda=0$ for $\lambda$ a generator of $H^1(Y;\mathbb{R})$ dual to the orientation of $H_1(Y;\mathbb{R})$. When the perturbation form $\delta$ is an exact 2-form, the small-perturbation Seiberg-Witten invariant $SW^0_Y(\mathfrak{s}_k)$ \cite{fintushel} is well defined for $Y$ with $b_1(Y)=1$
\[
  SW^0_Y(\mathfrak{s}_k) =
  \begin{cases}
    SW^+_Y(\mathfrak{s}_k)           & \text{if $k>0$} \\
    SW^-_Y(\mathfrak{s}_k)           & \text{if $k<0$} 
  \end{cases}.
\]
To see that the Seiberg-Witten invariant of the $spin^c$-structure $\xi_0$ over $X$ with parameter $(g_X,\pi^*(\delta)^+)$ is equal to the sum of the invariants $SW^0_Y(\mathfrak{s}_k)$ over all the $spin^c$-structures on $Y$ with parameter $(g_Y,\delta)$, we need to verify first that both sides are well defined under suitable choice of $(g_Y, \delta)$. Consider the exact perturbation $\delta=d\alpha\in \Omega^2(Y;\mathbb{R})$, the pull-back $\pi^*\delta=\pi^*(d\alpha)=d\pi^*\alpha\in \Omega^2(X;\mathbb{R})$ to $X$ is a $S^1$-invariant exact perturbation 2-form after projecting to self-dual component. Since $Y$ is three dimensional, we know the expected dimension of the moduli space is $0$ and in addition, we can find metric and exact perturbation $(g_Y, \delta)$ so that the moduli space is smooth without any reducible solution. In terms of the deformation complex associated with the gauge action and Seiberg-Witten equation
\[
0\rightarrow\Omega^0(Y;i\mathbb{R})\overset{\delta^0}{\rightarrow} \Omega^1(Y;i\mathbb{R})\oplus \Gamma(S)\overset{\delta^1}{\rightarrow} \Omega^1(Y;i\mathbb{R})\oplus \Gamma(S)\rightarrow 0,
\]
where the first map $\delta^0$ at a solution $(A_0,\Phi_0)$ is given by the derivative of gauge group action, 
\[
\delta^0(\gamma)=(2d\gamma,-\gamma\Phi_0),
\]
and the second map $\delta^1$ at a solution $(A_0,\Phi_0)$ is given by 
\[
\delta^1(a,\phi)=(*(da-\frac{1}{2}\sigma(\Phi_0,\phi)), D_{A_0+\alpha}\phi+\frac{1}{2}a\cdot\Phi_0),
\]
we know that the $H^0_{(A_0,\Phi_0)}=H^1_{(A_0,\Phi_0)}=H^2_{(A_0,\Phi_0)}=0$ for the complex above by our assumption on $(g_Y,\delta)$. We have a corresponding complex on $X$ 
\[
0\rightarrow \Omega^0(X; i\mathbb{R})\rightarrow \Omega^1(X;i\mathbb{R})\oplus\Gamma(S^+)\rightarrow \Omega^1(X;i\mathbb{R})\oplus\Gamma(S^-)\rightarrow 0
\]
at a solution $(A,\Phi)=\pi^*(A_0,\Phi_0)$ defined in a similar way. By Baldridge's theroem in \cite{baldridge}, we know that under the parameter $(g_X, \pi^*(\delta)^+)$, $\pi$ induces an isomorphism between $H^1_{(A_0,\Phi_0)}$ and $H^1_{(A,\Phi)}$, thus $H^1_{(A,\Phi)}=0$. When $X$ is a 4-manifold with free circle action, the expect dimension of the moduli space is 0 by direct computation. So at each irreducible solution $(A,\Phi)\in \mathcal{M}^*(X, g_X, \pi^*(\delta)^+)$, $H^0_{(A,\Phi)}=H^1_{(A,\Phi)}=H^2_{(A,\Phi)}=0$. Each equivalence class of solution in $\mathcal{M}^*(X, g_X, \pi^*(\delta)^+)$ is then an isolated point with smooth neighborhood modeled on the zero of the Kuranishi map $H^1_{(A, \Phi)}\rightarrow H^2_{(A,\Phi)}$. So there is a well-defined number (not an invariant) $SW_X(\xi_0,g_X, \pi^*(\delta)^+)$ defined by taking the algebraic count of points in $\mathcal{M}^*(X,g_X,\pi^*(\delta)^+)$. 

Now using the small-perturbation Seiberg-Witten invariant, the sum of Seiberg-Witten invariant over all $spin^c$-structures  $\mathfrak{s}_k$ on $Y$ is equal to 
\[
\sum_{k\in\mathbb{Z}} SW^0_Y(\mathfrak{s}_k)=\sum_{k\in\mathbb{Z}} a_{1+|k|}+2a_{2+|k|}+3a_{3+|k|}+\cdots,
\]
where $a_i$'s on the right hand side are coefficients of the normalized Alexander polynomial of $Y$. It's not hard to check that the right hand side is the $\Delta''_Y(1)$. Therefore, by the discussion above, we know that the Seiberg-Witten invariant of the $spin^c$-structure $\xi_0$ over $X$ with parameter $(g_X,\pi^*(\delta)^+)$ is equal to $\Delta''_Y(1)$. 

\section{Correction Term}

\subsection{Neck Stretching Operation} The correction term can be simplified by using the neck stretching operation discussed in detail in \cite{split}. Let $M$ be the 3-submanifold reperesenting the Poincare dual to the generator of $H^1(X;\mathbb{Z})$. The metric on $X$ induces a metric on $M$ by restriction. Assuming that the metric $g_X$ is a product in a neighborhood $[-\epsilon,\epsilon]\times M$, $\epsilon>0$. Consider the manifold "with long neck" 
\[
X_R=W\cup ([0,R]\times M),
\]
where $W$ is the cobordism obtained by cutting $X$ along $Y$. It's prove in \cite{split} that under certain assumptions, this long neck manifold $X_R$ with metric $g_R$ obtained by gluing metric $g_X|_W$ and product metric on the cylinder $[0,R]\times M$ can be used to compute the correction term. 
\begin{theorem}[\cite{split}]\label{neckstretch}
\[
\omega(X_R,g_R)=\ind \Dir^{+}(Z_{+}(M),g,\beta)+\sigma(Z)/8,
\]
where $Z_{+}(M)=Z\cup ([0,\infty)\times M)$ and $Z$ is a spin $4$-manifold with $\partial Z=Y$. It remains to check that the metric $g_X=\pi^*(g_Y)+\eta\otimes\eta$ used in computing the Seiberg-Witten invariant satisfies the following assumption from \cite{split}. 
\end{theorem}
\begin{assum}
The Dirac operator 
\[
\Dir^+(W_{\infty},g_{\infty}): L^2_1(W_{\infty}; S^+)\rightarrow L^2(W_{\infty};S^{-})
\]
is invertible, where $W_{\infty}=((-\infty,0]\times M)\cup W\cup ([0,+\infty)\times M)$ and $g_{\infty}$ is the metric on $W_{\infty}$ induced by $g_X$.
\end{assum}
This metric in the above assumption exists in the case when $X$ is an integral homology $S^1\times S^3$ by the Theorem 10.3 in \cite{split}. 
So in the following sections, we will focus on computing $\omega(X_R,g_R)=\ind \Dir^{+}(Z_{+}(M),g,\beta)+\sigma(Z)/8$. 
Using Atiyah-Patodi-Singer index theorem \cite{spectralI}, the correction term can be computed as 
\begin{align*}
\omega(X_R,g_R)&=\ind \Dir^{+}(Z_+(M),g,\beta)+\sigma(Z)/8\\
&= \bigg(\int_Z \hat{A}(p) -\frac{1}{2}h_{\Dir}-\frac{1}{2}\eta_{\Dir}(M)\bigg)+\frac{1}{8}\bigg(\int_Z L(p) -\eta_{ \sign}(M)\bigg)\\
&= -\frac{1}{2}h_{\Dir}-\frac{1}{2}\eta_{\Dir}(M)-\frac{1}{8}\eta_{\sign}(M). 
\end{align*}
Here $h_{\Dir}:=\dim \ker(\Dir^+|_M)$

\subsection{Eta Invariants of Dirac Operator}
Let $M$ be the restriction of the circle bundle $X\rightarrow Y$ to a closed surface $\Sigma$ which generates $H_2(Y;\mathbb{Z})$.  Equip $\Sigma$ with a constant sectional curvature metric $g_{\Sigma}$ such that $\vol(\Sigma)=\pi$. The induced metric on $M$ by restriction can be written as $g_M=\pi^*g_{\Sigma}\oplus \eta\otimes\eta$ and using this metric we can split $T^*M=\langle \eta\rangle\oplus \pi^*T^*\Sigma$ orthogonally. By rescaling the length of the fiber, we can form a family of metrics, parametrized by fiber length, 
\[g^r_M=\pi^*g_{\Sigma}\oplus \eta_r\otimes\eta_r,\]
where $\eta_r=r\eta$. For each $g_M^r$, there exists a Levi-Civita connection $\nabla^r$ which can be written in simple matrix form in well-chosen local frames. In \cite{etacircle}, the local orthonormal frame for $T^*M=\langle \eta\rangle\oplus \pi^*T^*\Sigma$ is chosen to be $(\eta_r,\eta^1,\eta^2)$ so that $\eta^i=\pi^*\theta^i, i=1,2$ where $\theta^i$ is a local orthonormal frame of $T^*\Sigma$ satisfying 
\[
d\theta^1=\kappa \theta^1\wedge\theta^2
\]
and 
\[
d\theta^2=0.
\]
The existence of this local frame comes from the classification of space forms. The connection 1-form under this local frame can be written in matrix form as 
\begin{equation}\omega_r= 
\begin{bmatrix}
0 & -r\eta^2 & -r\eta^1\\
r\eta^2 & 0 & r\eta_r-\kappa\eta^1\\
r\eta^1 & -r\eta_r+\kappa\eta^1 & 0
\end{bmatrix}.
\end{equation}

 In \cite{etacircle}, Nicolaescu studied the Dirac operators of type $\mathcal{D}_r$ associated to the connection with local connection 1-form of the above form when $r$ is small using the adiabatic limit technique. Note that Theorem \ref{neckstretch} holds when we use the partial rescaling metrics $g_X^r=\pi^*(g_Y)+r^2\eta\otimes\eta$ for arbitrarily small positive $r$. To see this, we use a result in \cite{dai} on the asymptotic behavior of spectrum of $\mathcal{D}_r$. Let $\{\lambda_{r}\}$ denote the spectrum of $\mathcal{D}_r$, by Dai's result of Theorem 1.5 in \cite{dai}, $\lambda_r$ is analytic on $r$ and either $|\lambda_{r}|\geq \frac{1}{r}\lambda_0\gg 0$ for $r$ sufficiently small or has the asymptotic formula below as 
 \begin{equation}\label{asymp}
 \lambda_{r}\sim \lambda_1+\lambda_2 r+...
 \end{equation}
 and when $\lambda_1\neq 0$, the spectrum of $\mathcal{D}_r$ satisfies 
 \[
 |\lambda_r|\geq \frac{1}{2}|\lambda_1| \text{ when } r \text{ is sufficiently close to 0}. 
 \]
 It remains to deal with the case when $\lambda_1=0$, in which case $\lambda_r$ decays at least linearly in $r$. It's sufficient to show that the first eigenvalue estimate Proposition 7.1 in \cite{split} holds uniformly for $r$ when $r$ is small. By the same idea in the proof of Proposition 7.1 and the result in \cite{dai}, the linear operator $T_{+,r}(\lambda, R): V_i(Y_2)\oplus V_i(Y_1)\rightarrow V_i(Y_2)\oplus V_i(Y_1)$ under the orthonormal eigenspinors of $\mathcal{D}_r$ has the matrix of the form
 \[
B_{i,r}=\frac{1}{(\lambda_{i,r}-\omega_{i,r})-(\lambda_{i,r}+\omega_{i,r})e^{2\omega_{i,r} R}} \begin{pmatrix}
\lambda(e^{2\omega_{i,r}R}-1) & -2\omega_{i,r}e^{\omega_{i,r} R}\\
-2\omega_{i,r}e^{\omega_{i,r}R} & \lambda(1-e^{2\omega_{i,r} R})
\end{pmatrix},
\]
where $\lambda_{i,r}$ denotes the $i$-th eigenvalue of $\mathcal{D}_r$ and $\omega_{i,r}=\sqrt{\lambda^2_{i,r}+\lambda^2}$. Therefore, the operator norm of $T_{+,r}(\lambda, R)$ can be estimated by 
\begin{equation}\label{norm1}
\bigg|\frac{2\omega_{i,r}e^{\omega_{i,r} R}}{(\lambda_{i,r}-\omega_{i,r})-(\lambda_{i,r}+\omega_{i,r})e^{2\omega_{i,r} R}}\bigg|\leq 2\cdot \frac{e^{\omega_{i,r}R}}{e^{2\omega_{i,r}R}-1}\leq 2\cdot \frac{e^{\lambda_{0,r}R}}{e^{2\lambda_{0,r}R}-1}
\end{equation}
and 
\begin{equation}\label{norm2}
\bigg|\frac{\lambda(e^{2\omega_{i,r}R}-1)}{(\lambda_{i,r}-\omega_{i,r})-(\lambda_{i,r}+\omega_{i,r})e^{2\omega_{i,r} R}}\bigg|\leq \frac{\lambda}{\omega_{i,r}}\leq \frac{\lambda}{\lambda_{0,r}},
\end{equation}
where $\lambda_{i,r}$ denotes the $i$-th eigenvalue of $\mathcal{D}_r$ and $\lambda$ denotes an eigenvalue of $\mathcal{D}_r$. Then we can prove a adiabatic version of Lemma 7.2 (3) in \cite{split}. 
\begin{prop}
For any $\epsilon>0$, there exists polynomials  $R_0(r)>0$ and $\epsilon_2(r)>0$ in $r$ such that, for any $R\geq R_0(r)$ and $0\leq \lambda<\epsilon_2(r)$, 
\[
|T_{\pm, r}(\lambda, R)|<\epsilon. 
\]
\end{prop}
\begin{proof}
Using (\ref{norm1}) and (\ref{norm2}), and the asymptotic formula (\ref{asymp}), we have that since $\lambda_{0,r}$ decays to $\lambda_1$ (either 0 or nonzero) at the rate of polynomial $P(r)$ by (\ref{asymp}), then by (\ref{norm2}), we have $\lambda/\lambda_{0,r}\leq \epsilon_2$, which implies that $\lambda\leq \epsilon_2\cdot P(r)$. By (\ref{norm1}), we have $\lambda_{0,r}R\geq R_0$, so $R\geq R_0/\lambda_{0,r}=R_0/P(r)$. 
\end{proof}
Under the assumption that the Dirac operator on the base satisfies $\ker D_Y=0$, then by Theorem 1.5 in \cite{dai}, we know $\lambda_1\in spec(D_Y\otimes \ker D_{S^1})$, thus $\lambda_1\neq 0$ in (\ref{asymp}). Under this assumption, the polynomials $\epsilon_2(r)$ and $R_0(r)$ can be chosen to be independent of $r$, and furthermore the first eigenvalue estimate is uniform in the fiber length $r$. See \cite{split}. 
\begin{prop}\label{first}
Assuming that the spin Dirac operator 
\[
\mathcal{D}^+_r: L^2_2(W_\infty; S^+)\rightarrow L^2_1(W_\infty; S^-)
\]
is an isomorphism for each $r$ and the Levi-Civita Dirac operator on the base satisfies $\ker D_Y=0$. Then there exists constants $R_0>0$ and $\epsilon_1>0$ such that for any $R\geq R_0$, the operator 
\[
\Delta_R=\mathcal{D}^-\mathcal{D}^+: L^2_2(X_R; S^+)\rightarrow L^2(X_R; S^+)
\]
has no eigenvalues in the interval $[0,\epsilon_1^2)$. 
\end{prop}
 Using Proposition \ref{first}, we can see Theorem \ref{neckstretch} holds for $g_X^r$ with arbitrarily small $r>0$ by checking the proof in \cite{split}. In step 6 of the proof in \cite{split},
 \[
 K: L^2_1(Z)\oplus (\bigoplus L^2_1(W_i))\rightarrow L^2(Z)\oplus (\bigoplus L^2(W_i))\oplus (\bigoplus V_{-}(M_i^-))\oplus (\bigoplus V_{+}(M_i^+))
 \]
 sending $\phi_0\oplus (\phi_1,\phi_2,...)$ to 
 \[
 0\oplus 0\oplus (-e^{R\mathcal{D}}\pi_{-}\phi_1|_{M_1^+}, -e^{R\mathcal{D}}\pi_{-}\phi_2|_{M_2^+},...)\oplus (-e^{-R\mathcal{D}}\pi_{-}\phi_0|_{M_1^-}, -e^{-R\mathcal{D}}\pi_{-}\phi_1|_{M_2^-},...).
 \]
 Now we can choose $r=r_i$ chosen above, then when $g_X$ is replaced with $g_X^{r_i}$, $\mathcal{D}|_M$ is replaced with $\mathcal{D}_{r_i}|_M$. If the minimum absolute value of eigenvalues of $\mathcal{D}_r$ is uniformly bounded below by $\epsilon_1>0$ for all sufficiently small $r>0$, then we can find a sufficiently large $R$ so that $e^{-R\epsilon_1}$ is sufficiently small, so the same argument works. 
 The following theorem gives a formula of the $\eta(\mathcal{D}_r)$. 
\begin{theorem}[\cite{etacircle}]
For all $0<r\ll r_0$, we have 
\begin{equation}
 \frac{1}{2}\eta(\mathcal{D}_r)=\frac{l}{12}-\sign(l)h_{1/2}+\frac{l}{12}(l^2r^4-\chi r^2) \label{dir}.
\end{equation}

\end{theorem} 
Here $h_{1/2}$ is the dimension of global holomorphic sections of $K_{\Sigma}^{1/2}$, the square root of canonical bundle over $\Sigma$ and $l$ is the Euler number of the circle bundle $M\rightarrow\Sigma$. The proof in \cite{etacircle} by Nicolaescu is done by studying the variation of $\eta(\mathcal{D}_r)$ as follows: Let $\xi_r=\frac{1}{2}(\eta(\mathcal{D}_r)+h(\mathcal{D}_r))$ where $h(\mathcal{D}_r)=\text{dim ker}(\mathcal{D}_r|M)$, then by Atiyah-Patodi-Singer index theorem, we can get a variation formula for $\xi_r$ in terms of spectral flow by studying the Dirac operator $\mathcal{D}_u$ on cylinder $[0,1]\times M$ equipped with metric $g=du^2\oplus g_{r(u)}$, $u$ is a coordinate on $[0,1]$, and $\nabla$ is the Levi-Civita connection of $g$. We have 
\[
\xi_{r_1}-\xi_{r_0}=SF(\mathcal{D}_{r(u)})+\int_{[0,1]\times M} \widehat{A}(\nabla).
\]
According to \cite{surg}, $\mathcal{D}_{r(u)}$ can be chosen to be invertible for each $u$, so the term $SF(\mathcal{D}_{r(u)})=0$. The remaining term
\[
\int_{[0,1]\times M} \widehat{A}(\nabla) 
\]
can be explicitly computed by using Chern-Simons transgression form
\[
T\widehat{A}(\nabla^{r(0)},\nabla^{r(1)})=\frac{d+1}{2}\int_0^1\widehat{A}(\omega,\Omega_t) dt,
\]
where $\omega=\nabla^{r(0)}-\nabla^{r(1)}$ and $\Omega_t$ is curvature form of $\nabla^{r(0)}+t\omega$. We have 
\[
\int_{[0,1]\times M} \widehat{A}(\nabla) = \int_M T\widehat{A}(\nabla^{r(0)},\nabla^{r(1)})=\frac{d+1}{2}\int_0^1\widehat{A}(\omega,\Omega_t)\,dt,
\]
which follows from a general lemma below
\begin{lem}
Let $F:\mathfrak{g}\times\mathfrak{g}\times\cdots\times\mathfrak{g}\rightarrow \mathbb{R}$ be a k-linear function on Lie algebra of $G$ and $F$ is invariant under adjoint action of $G$ on $g$. Given a linear path of connection 1-form $\omega_t=\omega_0+t\alpha$ on a principal $G$-bundle $P$, $\Omega_t=d\omega_t+\omega_t\wedge\omega_t$ is the curvature 2-form of $\omega_t$, then 
\[
\frac{d}{dt} F(\Omega_t,...,\Omega_t)=k dF(\alpha,\Omega_t,...,\Omega_t).
\]
\end{lem}
\begin{proof}
By definition
\begin{align*}
\Omega_t&=d\omega_t+\omega_t\wedge\omega_t\\
&=d\omega_0+td\alpha+(\omega_0+t\alpha)\wedge(\omega_0+t\alpha)\\
&=\Omega_0+td\alpha+t\omega_0\wedge\alpha+t\alpha\wedge\omega_0+t^2\alpha\wedge\alpha
\end{align*}
\begin{align*}
d\Omega_t&=d\omega_t\wedge\omega_t-\omega_t\wedge d\omega_t\\
&=(\Omega_t-\omega_t\wedge\omega_t)\wedge\omega_t-\omega_t\wedge(\Omega_t-\omega_t\wedge\omega_t)\\
&=[\Omega_t,\omega_t].
\end{align*}
Using linearity of $F$, we have $\frac{d}{dt}F(\Omega_t,...,\Omega_t)=k F(d\alpha+[\omega_t,\alpha],\Omega_t,...,\Omega_t)$, and \[dF(\alpha,\Omega_t,...,\Omega_t)=F(d\alpha,\Omega_t,...,\Omega_t)+(k-1)F(\alpha,[\omega_t,\Omega_t],...,\Omega_t).\]
Since $F$ is invariant under adjoint action, \[F([\omega_t,\alpha],\Omega_t,...,\Omega_t)-(k-1)F(\alpha,[\omega_t,\Omega_t],...,\Omega_t)=0,\]
it's immediate to get 
\[
\frac{d}{dt} F(\Omega_t,...,\Omega_t)=k dF(\alpha,\Omega_t,...,\Omega_t).
\]
\end{proof}


\subsection{Eta Invariants of Signature Operator}
In \cite{ouyang}, Ouyang computed the $\eta$-invariant of signature operator for circle bundles over surface $\Sigma$. In fact, he proved a more general theorem when $\Sigma$ is orbifold. 
\begin{theorem}[\cite{ouyang}]
Let $p: E\rightarrow\Sigma$ be a complex line bundle over surface $\Sigma$. Equip the fiber with metric $\widetilde{g}$ and let $\widetilde{\nabla}$ be a $\widetilde{g}$ preserving connection in $E$. Assume the curvature $\widetilde{R}$ is constant on $F$. Then the $\eta$-invariant of the circle bundle of radius $r$ is given by 
\begin{equation}
    \eta(S_r E)=\frac{2}{3}l\bigg\{ \frac{\pi r^2}{\vol(\Sigma)}\chi-(\frac{\pi r^2}{\vol(\Sigma)})^2 l^2\bigg\}+\frac{1}{3}l-\sign(l) \label{sig},
\end{equation}
where $l$ is the Euler number of the line bundle $E\rightarrow\Sigma$, $\chi$ is the Euler characteristic of $\Sigma$. 
\end{theorem}
We can check that the corresponding disk bundle of the circle bundle $M\rightarrow\Sigma$ equipped with connection $\eta$ and the metric $g_M=g_F\oplus\pi^*g_\Sigma=\eta\otimes\eta\oplus\pi^*g_\Sigma$ satisfies the conditions of the theorem above. First extend the metric from $M\rightarrow\Sigma$ to its disk bundle $E\rightarrow\Sigma$ by setting 
\[
g_E=dr^2+r^2 g_F + \pi^*g_\Sigma=dr^2+r^2\eta\otimes\eta+\pi^*g_\Sigma.
\]
The connection $\widetilde{\nabla}$ can be defined to be of the form 
\[
\widetilde{\nabla}=d\oplus\pi^*(\nabla^\Sigma),
\]
where $\nabla^{\Sigma}$ is the Levi-Civita connection of $g_{\Sigma}$. 
In fact, for any local vector fields $X,Y,Z$ on the fiber of $E\rightarrow\Sigma$, we have 
\begin{eqnarray*}
\widetilde{\nabla}_Z (\eta\otimes\eta(X,Y))&=&(d+i\eta)(Z)(\eta(X)\eta(Y))\\
&=& (Z\eta(X))\eta(Y)+\eta(X)(Z\eta(Y))+2i\eta(Z)\eta(X)\eta(Y)\\
&=& \eta(\widetilde{\nabla}_ZX)\eta(Y)+\eta(X)\eta(\widetilde{\nabla}_ZY)\\
&=& \eta\otimes\eta(\widetilde{\nabla}_ZX,Y)+\eta\otimes\eta(X,\widetilde{\nabla}_ZY).
\end{eqnarray*}
Therefore, we can see that $\widetilde{\nabla}$ is compatible with the fiber metric. The curvature tensor of $\widetilde{\nabla}$, $\widetilde{R}$ is pulled back from the curvature tensor $R$ of $\nabla^\Sigma$, so it is invariant along the fiber. 

\section{Result}
It's not hard to see from \ref{dir} and \ref{sig} that 
\begin{align}\label{eq: correc}
\omega(X,g_X,\beta) = -\frac{1}{2}h_{\Dir}-\frac{1}{2}\eta_{\Dir}(M, g|_M)-\frac{1}{8}\eta_{sign}(M)=-\frac{1}{2}h_{\Dir}+h_{1/2}
\end{align}
Note here in $\omega(X,g_X,\beta)$ we use the Levi-Civita connection of $g_X$ to define the $\eta$-invariant of Dirac operator, however, in the definition of Seiberg-Witten invariant, the connection we used is circle bundle compatible connection of the form $\tilde{\nabla}=d\oplus\pi^*(\nabla^{Y})$. The idea to solve this problem is to consider a path of connections $\nabla^t, t\in [0,1]$ connecting the Levi-Civita connection and the bundle compatible connection $\tilde{\nabla}$ such that $\nabla^t$ is compatible with $g_X$ for each $t\in [0,1]$, the associated Dirac operators $D_A^{r,t}$ at time $t$ can be viewed as a compact perturbation of $D_A^{r,0}$, so have the same index. 


Consider a  path of connections $\nabla^t$ by generalizing the method in \cite{etacircle} to $4$-manifolds: first define a sequence of bundle metrics parameterized by the length of fiber  $g_X^{(r)}=r^2\eta\otimes\eta\oplus \pi^*g_Y$ where $\eta$ be the globally defined connection 1-form of length 1 with respect to the metric $g_X=g_X^{(1)}$. Then we complete $r\eta$ to form a local orthonormal coframe of the form $\{e^0=r\eta, e^1, e^2, e^3\}$ and let $\{e_0,e_1,e_2,e_3\}$ be the corresponding local dual orthonormal frame with respect to the metric $g_X^{(r)}$. We define as in \cite{etacircle} a family of bundle maps $L_t: TX\rightarrow TX$ locally by 
\[
e_0\rightarrow t e_0, \quad e_i\rightarrow e_i, i=1,2,3
\]
where $e_0$ is the vector field of the free circle action defined earlier. $L_t$ defines an isometry from $(TX,g_X^{(rt)})$ to $(TX,g_X^{(r)})$ for $r>0$ and $t\in (0,1]$. Now the connection defined by 
\[
\nabla^{r,t}=L_t\nabla^{rt}L_t^{-1}
\]
is compatible with $g_X^{(r)}$. To see this, let $X, Y, Z$ be local vector field on $X$ and compute the derivative of $g^{(r)}(Y,Z)$ in the direction $X$. 
\begin{eqnarray*}
Xg^{(r)}(Y,Z) &=& Xg^{(rt)}(L_t^{-1}Y, L_t^{-1}Z)\\
&=& g^{(rt)}(\nabla^{(rt)}_X L_t^{-1}Y, L_t^{-1}Z)+g^{(rt)}(L_t^{-1}Y, \nabla^{(rt)}_X L_t^{-1}Z)\\
&=& g^{(r)}(L_t\nabla^{(rt)}_XL_t^{-1}Y, Z)+g^{(r)}(Y, L_t\nabla^{(rt)}_XL_t^{-1}Z).    
\end{eqnarray*}

We will choose $\nabla^{r,t}, t\in [0,1]$ as our path of connections. Using the local frame defined earlier, we can write down the matrix of connection 1-form $\omega$ as follows 
{\footnotesize
\[
\omega_{r,t}=
\begin{bmatrix}
        0   & ra^{(t)}_{12}e^2+ra^{(t)}_{13}e^3 & -ra^{(t)}_{12}e^1+ra^{(t)}_{23}e^3 & -ra^{(t)}_{13}e^1-ra^{(t)}_{23}e^2 \\
    -ra^{(t)}_{12}e^2-ra^{(t)}_{13}e^3 & 0 & -ra^{(t)}_{12}e^0+\omega^1_2 & -ra^{(t)}_{13}e^0+\omega^1_3 \\
    ra^{(t)}_{12}e^1-ra^{(t)}_{23}e^3 & ra^{(t)}_{12}e^0-\omega^1_2 & 0 & -ra^{(t)}_{23}e^0+\omega^2_3\\
    ra^{(t)}_{13}e^1+ra^{(t)}_{23}e^2 & ra^{(t)}_{13}e^0-\omega^1_3 & ra^{(t)}_{23}e^0-\omega^2_3 & 0 
\end{bmatrix},
\]}
where $a_{ij}^{(t)}=ta_{ij}$ and $a_{ij}$ is defined by  
\[
d\eta=e^1\wedge(a_{12}e^2+a_{13}e^3)+e^2\wedge(-a_{12}e^1+a_{13}e^3)+e^3\wedge(-a_{13}e^1-a_{23}e^2).
\]
The connection 1-form matrix of $\widetilde{\nabla}$ is 
\[
\widetilde{\omega}=
\begin{bmatrix}
        0   & 0 & 0 & 0 \\
    0 & 0 & \omega^1_2 & \omega^1_3 \\
    0 & -\omega^1_2 & 0 & \omega^2_3\\
    0 & -\omega^1_3 & -\omega^2_3 & 0 
\end{bmatrix}.
\]
We can see from the connection matrix above that $\nabla^{r,t}\rightarrow\widetilde{\nabla}$ as $t\rightarrow 0$. When $t=1$, $L_t=id$, $\nabla^{r,t}$ is just the Levi-Civita connection of $g_X^{(r)}$. The path of corresponding Dirac operators can be written down as  
\begin{lem}
\[
D^{r,t}_A=D_A-\frac{1}{2}r^2t^2\sigma (\eta\wedge d\eta),
\]
where $D^{r,t}_A$ is the Dirac operator associated to the Levi-Civita connection $\nabla^{r,t}$ and $D_A$ is the Dirac operator associated to the connection $\widetilde{\nabla}$. 
\end{lem}
\begin{proof}
The proof is essentially the same as the proof given in \cite{baldridge}. It follows by writing down the local connection 1-form matrix $\omega^{r,t}$ and $\widetilde{\omega}$ for $\nabla^{r,t}$ and $\widetilde{\nabla}$ respectively using a local frame as we did above and take the difference 1-form  $\omega=\omega^{r,t}-\widetilde{\omega}\in \Omega^1(\mathfrak{so}(T^*X))$. Then the local difference of the two corresponding Dirac operators $D^{r,t}_A$ and $D_A$ can be written as the Clifford multiplication by $\omega$
\[
D^{r,t}_A-D_A=\sigma(\omega)
\]
here $\omega\in \Omega^1(\mathfrak{so}(T^*X))\cong \Omega^1(\Lambda^2 T^*X)$, where the latter is the space of 1-forms with value in the exterior square of $T^*X$. Using the isomorphism 
\[
(a^k_j)\longmapsto \frac{1}{2}\sum_{j<k} a^j_ke^j\wedge e^k
\]
from $\mathfrak{so}(4)$ to $\Lambda^2 T^*X$, we can write $\omega$ as an element in $\Omega^1(\Lambda^2 T^*X)$ 
\[
\omega=\frac{1}{2}\sum_{i=1}^3 e^i\otimes (rt\eta)\wedge \iota_{e_i}(d(rt\eta))+\frac{1}{2}rt\eta\otimes d(rt\eta),
\]
then 
\[
\sigma(\omega)=-\frac{1}{2}r^2t^2\sigma(\eta\wedge d\eta).
\]
\end{proof}

It remains to show that the index of $D_A^{r,t}$ is unchanged along the path $t=0$ to $t=1$. As we can see from the lemma above, $D_A^{r,t}$ can be thought of as zero order perturbation of $D_A$ and by the theory in compact operator, it's sufficient to prove the following lemma
\begin{prop}
If $\omega\in \Omega^1(\Lambda^2T^*X)$, and $i$ is the Sobolev embedding $L^2_1(W^\infty; S^-)\subset L^2(W^\infty; S^-)$, $\sigma(\omega)$ is the Clifford multiplication, then the composition
\[
i\circ\sigma(\omega): L^2_1(W^\infty; S^+)\rightarrow L^2(W^\infty; S^-)
\]
is compact. 
\end{prop}
\begin{proof}
The Sobolev inequality may fail for non-compact manifold. So $i$ may not be a compact operator in general. We use instead the Laplace-Fourier transform introduced in \cite{liftof}. Consider the following diagram, 
\begin{center}
\begin{tikzcd}
L^2_1(W^\infty; S^+) \arrow[r, "\Psi"] \arrow [d,"\mathcal{F}"]
& L^2(W^\infty; S^-) \arrow[d, "\mathcal{F}"] \\
L^2_1(X; S^+) \arrow[r, "\widehat{\Psi}"]
& L^2(X;S^-)
\end{tikzcd}
\end{center}
Here $\Psi$ is the composition $i\circ\sigma(\omega)$ defined above, and $\mathcal{F}$ is the Laplace-Fourier transform. To prove the compactness of $\Psi$, consider a bounded sequence of sections $\{u_k\}\in L^2_1(W^\infty; S^+)$ and we need to prove $\{\Psi(u_k)\}$ has a convergent subsequence. To show this, we apply Laplace-Fourier transform to $\{u_k\}$ and get a sequence $\{\mathcal{F}(u_k)\}$ of sections in $L^2_1(X; S^+)$, which is bounded by proposition 4.1 in \cite{liftof}. By direct computation, $\widehat{\Psi}$ has the same form as $\Psi=i\circ\sigma(\omega): L^2_1(X; S^+)\rightarrow L^2(X;S^-)$, which is a compact operator when $X$ is compact by the Rellich theorem. So $\{\widehat{\Psi}(\mathcal{F}(u_k))\}$ has a convergent subsequence. We obtain a corresponding subsequence by taking the inverse transform as is defined in \cite{liftof}
\[
v_{k}(x+n)=\frac{1}{2\pi i}\int_{I(\nu)}e^{-\mu (f(x)+n)} \widehat{\Psi}(\mathcal{F}(u_{k}))(x)\,d\mu.
\]
We can prove that $v_k(x)$ is convergent by showing that the inverse Laplace transform 
\[
L^2(X;S^-)\rightarrow L^2(W^\infty;S^-)
\]
is bounded. This can be seen by 
\begin{align*}
    \int_{W^\infty} |g|\cdot|v_k|\,dx &= \frac{1}{2\pi i}\int_{W^\infty} |g|\cdot\bigg|\int_{I(\nu)}e^{-\mu (f(x)+n)}\widehat{\Psi}(\mathcal{F}(u_{k}))(x)\bigg|\,dx\\
    &\leq \frac{1}{2\pi i}\int_{W^{\infty}}|g|\cdot\|\widehat{\Psi}(\mathcal{F}(u_{k}))(x)\|_{L^2(X)}\cdot\bigg(\int_{I(\nu)}|e^{-\mu (f(x)+n)}|^2\bigg)^{1/2}\\
    &\leq \frac{1}{2\pi i}\|g\|_{L^2(W^\infty)}\cdot\bigg(\int_{W^\infty}\int_{I(\nu)}|e^{-\mu (f(x)+n)}|^2\bigg)^{\frac{1}{2}}\cdot \|\widehat{\Psi}(\mathcal{F}(u_{k}))\|_{L^2(X)}
\end{align*}
and the fact that the integral 
\[
\int_{W^\infty}\int_{I(\nu)}|e^{-\mu (f(x)+n)}|^2<\infty
\]
\end{proof}

In particular, we can use the above result to prove the correction term (\ref{eq: correc}) is 0. 
\begin{align*}
    \omega(X, g_X, \beta)=-\frac{1}{2}h_{\Dir}+h_{1/2}=-\frac{1}{2}h_{\dir}+h_{1/2}
\end{align*}
and in \cite{etacircle}, Nicolescu claimed the last term is 0. 

\bibliographystyle{amsplain}

\end{document}